\documentclass[10pt, reqno]{amsart}
\usepackage{amsmath, amsthm, amscd, amsfonts, amssymb, graphicx, color}
\usepackage[bookmarksnumbered, colorlinks, plainpages]{hyperref}

\hypersetup{colorlinks=true,linkcolor=red, anchorcolor=green, citecolor=cyan, urlcolor=red, filecolor=magenta, pdftoolbar=true}

\newtheorem{theorem}{Theorem}[section]
\newtheorem{lemma}[theorem]{Lemma}

\newtheorem{corollary}[theorem]{Corollary}
\theoremstyle{definition}
\newtheorem{definition}[theorem]{Definition}
\newtheorem{example}[theorem]{Example}

\theoremstyle{remark}
\newtheorem{remark}[theorem]{Remark}

\numberwithin{equation}{section}

\begin{document}

\setcounter{page}{1}

\title[Generalized fractional integrals on Morrey spaces]
{Some new estimates for generalized fractional integrals associated with operators on Morrey spaces}

\author[Hua Wang]{Hua Wang}

\address{School of Mathematics and Information Science, Xiangnan University,
Chenzhou 423000, P. R. China}
\email{\textcolor[rgb]{0.00,0.00,0.84}{wanghua@pku.edu.cn}}

\dedicatory{Dedicated to the memory of Li Xue}



\date{\today}

\begin{abstract}
Let $\mathcal{L}$ be the infinitesimal generator of an analytic semigroup $\big\{e^{-t\mathcal L}:t>0\big\}$ on $L^2(\mathbb R^n)$ with Gaussian upper bounds, and suppose that $\mathcal{L}$ has a bounded holomorphic functional calculus on $L^2(\mathbb R^n)$. For given $0<\alpha<n$, let $\mathcal L^{-\alpha/2}$ be the generalized fractional integral associated with $\mathcal{L}$, which is given by
\begin{equation*}
\mathcal L^{-\alpha/2}(f)(x):=\frac{1}{\Gamma(\alpha/2)}\int_0^{+\infty}e^{-t\mathcal L}(f)(x)t^{\alpha/2-1}dt,
\end{equation*}
where $\Gamma(\cdot)$ is the usual gamma function. In the limiting Sobolev case $\lambda=n-\alpha p$ and $1\leq p<n/{\alpha}$, the author proves that the operator $\mathcal{L}^{-\alpha/2}$ is bounded from the Morrey space $M^{p,\lambda}(\mathbb R^n)$ into $\mathrm{BMO}_{\mathcal{L}}(\mathbb R^n)$, and is bounded from the vanishing Morrey space $VM^{p,\lambda}(\mathbb R^n)$ into $\mathrm{VMO}_{\mathcal{L}}(\mathbb R^n)$, where $\mathrm{BMO}_{\mathcal{L}}(\mathbb R^n)$ and $\mathrm{VMO}_{\mathcal{L}}(\mathbb R^n)$ are the spaces of bounded mean oscillation and vanishing mean oscillation associated with the operator $\mathcal{L}$, respectively. As a consequence, the author obtains that the operator $\mathcal{L}^{-\alpha/2}$ is bounded from $L^{p,\infty}(\mathbb R^n)$ into $\mathrm{BMO}_{\mathcal{L}}(\mathbb R^n)$ when $p=n/{\alpha}$ and $0<\alpha<n$. The proofs are based on pointwise kernel estimates of the operators $\mathcal L^{-\alpha/2}$ and $(I-e^{-t\mathcal L})\mathcal{L}^{-\alpha/2}$ for $0<\alpha<n$.
\end{abstract}

\maketitle

\section{Introduction and preliminaries}
\label{sec1}
Let $\mathbb R^n$ be the $n$-dimensional Euclidean space endowed with the Lebesgue measure $dx$ and the Euclidean norm $|\cdot|$. Let $\mathcal{L}$ be the infinitesimal generator of an analytic semigroup $\big\{e^{-t\mathcal L}:t>0\big\}$ on $L^2(\mathbb R^n)$ with Gaussian upper bounds on its heat kernel, and suppose that $\mathcal{L}$ has a bounded holomorphic functional calculus on $L^2(\mathbb R^n)$.  Then $\mathcal{L}$ is the linear operator on $L^2(\mathbb R^n)$ which generates an analytic semigroup $\big\{e^{-t\mathcal L}:t>0\big\}$ with kernel $\mathcal{P}_t(x,y)$ satisfying
\begin{equation*}
e^{-t\mathcal L}(f)(x):=\int_{\mathbb R^n}\mathcal{P}_t(x,y)f(y)\,dy, \quad t>0,
\end{equation*}
and there exist two positive constants $C$ and $A$ such that
\begin{equation}\label{G}
\big|\mathcal{P}_t(x,y)\big|\leq\frac{C}{t^{n/2}}\cdot e^{-A\frac{|x-y|^2}{t}}
\end{equation}
holds for all $x,y\in\mathbb R^n$ and all $t>0$. For any $0<\alpha<n$, the generalized fractional integral $\mathcal L^{-\alpha/2}$ associated with the operator $\mathcal{L}$ is defined by
\begin{equation}\label{gefrac}
\mathcal L^{-\alpha/2}(f)(x):=\frac{1}{\Gamma(\alpha/2)}\int_0^{+\infty}e^{-t\mathcal L}(f)(x)t^{\alpha/2-1}dt,\quad x\in\mathbb R^n.
\end{equation}
Let $\Delta$ be the Laplacian operator on $\mathbb R^n$, that is,
\begin{equation*}
\Delta:=\frac{\partial^2}{\partial x_1^2}+\frac{\partial^2}{\partial x_2^2}+\cdots+\frac{\partial^2}{\partial x_n^2}.
\end{equation*}
Note that if $\mathcal{L}=-\Delta$ is the Laplacian operator on $\mathbb R^n$, then $\mathcal L^{-\alpha/2}$ is exactly the classical fractional integral(or Riesz potential operator) $\mathcal{I}_{\alpha}$ of order $\alpha$($0<\alpha<n$), which is given by
\begin{equation*}
\mathcal{I}_{\alpha}(f)(x):=\frac{1}{\gamma(\alpha)}\int_{\mathbb R^n}\frac{f(y)}{|x-y|^{n-\alpha}}dy,\quad x\in\mathbb R^n,
\end{equation*}
where $\gamma(\alpha):=\frac{2^{\alpha}\pi^{n/2}\Gamma(\alpha/2)}{\Gamma({(n-\alpha)}/2)}$ and $\Gamma(\cdot)$ being the usual gamma function.

\begin{example}
The Gaussian upper bound \eqref{G} is satisfied by a large class of differential operators. Some of them are listed below:
\begin{enumerate}
\item Let $0\leq V(x)\in L^1_{\mathrm{loc}}(\mathbb R^n)$. The Schr\"{o}dinger operator with nonnegative potential $V(x)$ is defined by
\begin{equation*}
\mathcal{L}_1=-\Delta+V \quad \mbox{on}~~ \mathbb R^n.
\end{equation*}
The operator $\mathcal{L}_1$ is a self-adjoint positive definite operator. Hence it has a bounded functional calculus on $L^2(\mathbb R^n)$. By the well-known Feynman--Kac formula, we can see that the kernel $\mathcal{P}_t(x,y)$ of the semigroup $\big\{e^{-t\mathcal L_1}:t>0\big\}$ satisfies the following pointwise estimate
\begin{equation*}
0\leq \mathcal{P}_t(x,y)\leq\frac{1}{(4\pi t)^{n/2}}e^{-\frac{|x-y|^2}{4t}}.
\end{equation*}
That is, the Gaussian upper bound \eqref{G} is satisfied. For some more interesting results on the generalized fractional integral $\mathcal{L}_1^{-\alpha/2}$ generated by the operator $\mathcal{L}_1$,
we refer the reader to \cite{ak}, \cite{bui}, \cite{tang}, \cite{wang} and the references therein.
\item Let $\mathcal{A}=\mathcal{A}(x)$ be an $n\times n$ matrix of complex $L^{\infty}$-coefficients defined on $\mathbb R^n$. For $\xi=(\xi_1,\xi_2,\dots,\xi_n)\in \mathbb{C}^n$, we denote its complex conjugate $(\bar{\xi_1},\bar{\xi_2},\dots,\bar{\xi_n})$ by $\bar{\xi}$. Assume that this matrix satifies the following ellipticity (or ``accretivity") condition:
\begin{equation*}
\lambda|\xi|^2\leq \mathrm{Re}\, \mathcal{A}\xi\cdot\bar{\xi}\quad \mbox{and} \quad |\mathcal{A}\xi\cdot\bar{\zeta}|\leq\Lambda|\xi||\zeta|
\end{equation*}
for all $\xi,\zeta\in \mathbb{C}^n$ and for some constants $\lambda,\Lambda$ such that $0<\lambda\leq\Lambda<\infty$. Here we use the inner product notation
\begin{equation*}
\xi\cdot\bar{\zeta}=\xi_1\bar{\zeta_1}+\xi_2\bar{\zeta_2}+\cdots+\xi_n\bar{\zeta_n}.
\end{equation*}
Then we define the second-order divergence form operator as
\begin{equation*}
\mathcal{L}_2=-\mathrm{div}(\mathcal{A}\nabla)\quad \mbox{on}~~ \mathbb R^n,
\end{equation*}
which we interpret in the standard weak sense of maximal accretive operators via a sesquilinear form. The maximal accretivity condition ensures the existence of an analytic contraction semigroup on $L^2(\mathbb R^n)$ generated by $-\mathcal{L}_2$. In this setting, the complex elliptic operator $\mathcal{L}_2$ also has a bounded holomorphic functional calculus on $L^2(\mathbb R^n)$. It is known that when $\mathcal{A}$ has real entries, or when $n=1,2$ in the case of complex entries, the operator $\mathcal{L}_2$ generates an analytic semigroup $\big\{e^{-t\mathcal L_2}:t>0\big\}$ with heat kernel satisfying the Gaussian upper bound \eqref{G}(see, for example, \cite{d} and \cite{duong1}).
In general, when $n\geq3$ in the case of complex entries, the property
\eqref{G} on the heat kernel does not hold. For some more results on the generalized fractional integral $\mathcal{L}_2^{-\alpha/2}$ generated by the operator $\mathcal{L}_2$, we refer the reader to \cite{auscher2}, \cite{auscher}, \cite{chen} and the references therein.
\end{enumerate}
\end{example}
Since the semigroup $\big\{e^{-t\mathcal L}:t>0\big\}$ has a kernel $\mathcal{P}_t(x,y)$ which satisfies the Gaussian upper bound \eqref{G}, it is easy to check that there exists a constant $C>0$ such that
\begin{equation}\label{dominate1}
\big|\mathcal L^{-\alpha/2}(f)(x)\big|\leq C\cdot \mathcal{I}_{\alpha}(|f|)(x)
\end{equation}
holds for all $x\in\mathbb R^n$. In fact, if we denote the kernel of $\mathcal L^{-\alpha/2}$ by $\mathcal K_{\alpha}(x,y)$, then it follows from \eqref{gefrac} and Fubini's theorem that
\begin{equation*}
\begin{split}
\int_{\mathbb R^n}\mathcal K_{\alpha}(x,y)f(y)\,dy
&=\mathcal{L}^{-\alpha/2}(f)(x)\\
&=\frac{1}{\Gamma(\alpha/2)}\int_0^{+\infty}e^{-t\mathcal L}(f)(x)t^{\alpha/2-1}dt\\
&=\int_0^{+\infty}\bigg\{\frac{1}{\Gamma(\alpha/2)}\int_{\mathbb R^n}\mathcal{P}_t(x,y)f(y)\,dy\bigg\}t^{\alpha/2-1}dt\\
&=\int_{\mathbb R^n}\bigg\{\frac{1}{\Gamma(\alpha/2)}\int_0^{+\infty}\mathcal{P}_t(x,y)t^{\alpha/2-1}dt\bigg\}f(y)\,dy.
\end{split}
\end{equation*}
Then we have
\begin{equation}\label{kgamma}
\mathcal K_{\alpha}(x,y)=\frac{1}{\Gamma(\alpha/2)}\int_0^{+\infty}\mathcal{P}_t(x,y)t^{\alpha/2-1}dt,
\end{equation}
where $\mathcal{P}_t(x,y)$ is the kernel of $e^{-t\mathcal L}$. Thus, by using the Gaussian upper bound \eqref{G} and the expression \eqref{kgamma}, we can deduce that for any $0<\alpha<n$,
\begin{align}\label{kernelk}
\big|\mathcal K_{\alpha}(x,y)\big|
&\leq\frac{1}{\Gamma(\alpha/2)}\int_0^{+\infty}\big|\mathcal{P}_t(x,y)\big|t^{\alpha/2-1}dt\notag\\
&\leq C\cdot\int_0^{+\infty} e^{-A\frac{|x-y|^2}{t}}\cdot t^{\alpha/2-n/2-1}dt\notag\\
&\leq C\cdot\frac{1}{|x-y|^{n-\alpha}}\int_0^{+\infty} e^{-\nu}\cdot\nu^{n/2-\alpha/2-1}d\nu\notag\\
&\leq C\cdot\frac{1}{|x-y|^{n-\alpha}}.
\end{align}
This proves \eqref{dominate1} with $C>0$ independent of $f$ (see \cite{duong1} and \cite{mo}). Recall that, for any given $1\leq p<\infty$, the space $L^p(\mathbb R^n)$ is defined as the set of all integrable functions $f$ on $\mathbb R^n$ such that
\begin{equation*}
\|f\|_{L^p}:=\bigg(\int_{\mathbb R^n}|f(x)|^pdx\bigg)^{1/p}<+\infty,
\end{equation*}
and the weak space $L^{p,\infty}(\mathbb R^n)$ is defined as the set of all measurable functions $f$ on $\mathbb R^n$ such that
\begin{equation*}
\|f\|_{L^{p,\infty}}:=\sup_{\lambda>0}\lambda\cdot m\big(\big\{x\in\mathbb R^n:|f(x)|>\lambda\big\}\big)^{1/p}<+\infty.
\end{equation*}
In the sequel, for a Lebesgue measurable set $E\subset\mathbb R^n$, the $n$-dimensional Lebesgue measure of $E$ is denoted by $m(E)$. For the case $p=\infty$, $L^{\infty,\infty}(\mathbb R^n)$ will be taken to mean $L^\infty(\mathbb R^n)$, which is defined as the set of all measurable functions $f$ on $\mathbb R^n$ such that
\begin{equation*}
\|f\|_{L^\infty}:=\underset{x\in\mathbb R^n}{\mbox{ess\,sup}}\,|f(x)|<+\infty.
\end{equation*}
A real-valued locally integrable function $f$ on $\mathbb R^n$ is said to be in $\mathrm{BMO}(\mathbb R^n)$, the space of bounded mean oscillation, if
\begin{equation}\label{bmonorm}
\|f\|_{\mathrm{BMO}}:=\sup_{\mathcal{B}\subset\mathbb R^n}\frac{1}{m(\mathcal{B})}\int_{\mathcal{B}}|f(x)-f_{\mathcal{B}}|\,dx<+\infty,
\end{equation}
where the supremum is taken over all balls $\mathcal{B}$ contained in $\mathbb R^n$, and $f_{\mathcal{B}}$ denotes the mean value of $f$ over the ball $\mathcal{B}$, that is,
\begin{equation*}
f_{\mathcal{B}}:=\frac{1}{m(\mathcal{B})}\int_{\mathcal{B}}f(y)\,dy,
\end{equation*}
and $m(\mathcal{B})$ is the Lebesgue measure of $\mathcal{B}$. Modulo constant functions, the functional in \eqref{bmonorm} is a norm under which the space BMO is a Banach space. It is well known that for any $1<p<\infty$,
\begin{equation*}
\|f\|_{\mathrm{BMO}}\leq \bigg(\frac{1}{m(\mathcal{B})}\int_{\mathcal{B}}|f(x)-f_{\mathcal{B}}|^pdx\bigg)^{1/p}\leq C\|f\|_{\mathrm{BMO}}
\end{equation*}
with $C>0$ depending only on the dimension $n$ and $p$. The proof is based on H\"{o}lder's inequality and the famous John--Nirenberg's inequality(see \cite{john}). Let $\mathrm{BMO}^p(\mathbb R^n)$ be the set of all locally integrable functions $f$ satisfying
\begin{equation*}
\|f\|_{\mathrm{BMO}^p}:=\sup_{\mathcal{B}\subset\mathbb R^n}\bigg(\frac{1}{m(\mathcal{B})}\int_{\mathcal{B}}|f(x)-f_{\mathcal{B}}|^pdx\bigg)^{1/p}<+\infty.
\end{equation*}
Then the space $\mathrm{BMO}^p(\mathbb R^n)$ coincides with $\mathrm{BMO}(\mathbb R^n)$ and the corresponding norms are equivalent.
According to Sarason \cite{sarason}, a function $f\in \mathrm{BMO}(\mathbb R^n)$ is said to be in $\mathrm{VMO}(\mathbb R^n)$, the space of vanishing mean oscillation, if it satisﬁes the limiting condition
\begin{equation*}
\lim_{r\to 0^{+}}\sup_{x_0\in\mathbb R^n}\frac{1}{m(\mathcal{B})}\int_{\mathcal{B}}|f(x)-f_{\mathcal{B}}|\,dx=0,
\end{equation*}
where $\mathcal{B}=B(x_0,r)$ denotes the open ball centered at $x_0\in\mathbb R^n$ with radius $r>0$. For brevity, in the sequel we use the notation
\begin{equation*}
\eta(f;r):=\sup_{x_0\in\mathbb R^n}\frac{1}{m(\mathcal{B})}\int_{\mathcal{B}}|f(x)-f_{\mathcal{B}}|\,dx,
\end{equation*}
for any $r>0$, when $f\in \mathrm{BMO}(\mathbb R^n)$. $\eta(f;r)$ is called the VMO modulus of $f$. Note that when $f\in \mathrm{BMO}(\mathbb R^n)$ and $\mathcal{B}=B(x_0,r)$,
\begin{equation*}
\|f\|_{\mathrm{BMO}}=\sup_{x_0\in\mathbb R^n,r>0}\frac{1}{m(\mathcal{B})}\int_{\mathcal{B}}|f(x)-f_{\mathcal{B}}|\,dx<+\infty.
\end{equation*}
When $f\in \mathrm{VMO}(\mathbb R^n)$,
\begin{equation*}
\lim_{r\to 0^{+}}\eta(f;r)=0.
\end{equation*}
It is well known that the classical fractional integral operator $I_{\alpha}$ of order $\alpha$ plays an important role in harmonic analysis, potential theory and PDEs, particularly in the study of smoothness properties of functions. Let $0<\alpha<n$ and $1<p<q<\infty$. The classical Hardy--Littlewood--Sobolev theorem states that $\mathcal{I}_{\alpha}$ is bounded from $L^p(\mathbb R^n)$ to $L^q(\mathbb R^n)$, where $1<p<n/{\alpha}$ and $1/q=1/p-\alpha/n$. However, in the limiting Sobolev case $p=n/{\alpha}$ and $0<\alpha<n$, $\mathcal{I}_{\alpha}$ is NOT bounded from $L^p(\mathbb R^n)$ to $L^{\infty}(\mathbb R^n)$. Instead, in this case, it is actually bounded from $L^{n/{\alpha}}(\mathbb R^n)$ into $\mathrm{BMO}(\mathbb R^n)$(see, for example, \cite{grafakos} and \cite{stein}). In view of \eqref{dominate1}, we can generalize the above classical result from $-\Delta$ to the more general operator $\mathcal{L}$ with Gaussian upper bound \eqref{G}, and obtain the $(L^p,L^q)$ boundedness of the operator $\mathcal L^{-\alpha/2}$ under the same conditions.

On the other hand, the classical Morrey space ${M}^{p,\lambda}(\mathbb R^n)$ was introduced by Morrey to study
the local behavior of solutions to second order elliptic partial differential equations. For the properties and applications of classical Morrey spaces, we refer the reader to the books \cite{adams1}, \cite{hakim1} and \cite{hakim2}. Let us now recall the definitions of the Morrey space and vanishing Morrey space.
\begin{definition}
Let $n\in \mathbb{N}$, $1\leq p<\infty$ and $0\leq\lambda\leq n$. We say that a locally integrable function $f$ on $\mathbb R^n$ belongs to the Morrey space $M^{p,\lambda}(\mathbb R^n)$, if
\begin{equation*}
\begin{split}
\big\|f\big\|_{{M}^{p,\lambda}}:=&\sup_{\mathcal{B}\subset\mathbb R^n}\frac{1}{m(\mathcal{B})^{\lambda/{pn}}}
\bigg(\int_{\mathcal{B}}|f(x)|^p\,dx\bigg)^{1/{p}}\\
=&\sup_{\mathcal{B}\subset\mathbb R^n}\frac{1}{m(\mathcal{B})^{\lambda/{pn}}}
\big\|f\cdot\chi_{\mathcal{B}}\big\|_{L^p}<+\infty.
\end{split}
\end{equation*}
\end{definition}
\begin{remark}
The space $M^{p,\lambda}(\mathbb R^n)$ becomes a Banach space with the norm $\|\cdot\|_{{M}^{p,\lambda}}$. Moreover, for $\lambda=0$, the Morrey space $M^{p,0}(\mathbb R^n)$ coincides with the Lebesgue space $L^p(\mathbb R^n)$. Meanwhile, if $0<\lambda<n$ and $1\leq q<p$, then $M^{q,\lambda}(\mathbb R^n)$ is strictly larger than $L^p(\mathbb R^n)$ when $1-q/p=\lambda/n$. For instance, it can be shown that the function $f(x):=|x|^{-n/p}$ belongs to $M^{q,\lambda}(\mathbb R^n)$, but it is clearly not in $L^p(\mathbb R^n)$. For $\lambda=n$, we have $M^{p,n}(\mathbb R^n)=L^\infty(\mathbb R^n)$ by the Lebesgue differentiation theorem.
\end{remark}
Let $\mathcal{B}=B(x_0,r)$ be a ball centered at $x_0\in \mathbb R^n$ and with radius $r>0$. Note that
\begin{equation*}
\big\|f\big\|_{{M}^{p,\lambda}}\approx\sup_{x_0\in \mathbb R^n,r>0}\frac{1}{r^{\lambda/{p}}}
\bigg(\int_{\mathcal{B}}|f(x)|^p\,dx\bigg)^{1/{p}}
=\sup_{x_0\in \mathbb R^n,r>0}\bigg(\frac{1}{r^{\lambda}}\int_{\mathcal{B}}|f(x)|^p\,dx\bigg)^{1/{p}}.
\end{equation*}
For simplicity, in the sequel we use the notation
\begin{equation*}
\mathcal{M}_{p,\lambda}(f;x_0,r):=\frac{1}{m(\mathcal{B})^{\lambda/{pn}}}
\bigg(\int_{\mathcal{B}}|f(x)|^p\,dx\bigg)^{1/p}.
\end{equation*}
\begin{definition}
Let $n\in \mathbb{N}$, $1\leq p<\infty$ and $0<\lambda<n$. The vanishing Morrey space $VM^{p,\lambda}(\mathbb R^n)$ is defined as the set of all functions $f\in M^{p,\lambda}(\mathbb R^n)$ satisfying the limiting condition
\begin{equation*}
\lim_{r\to 0^{+}}\sup_{x_0\in \mathbb R^n}\mathcal{M}_{p,\lambda}(f;x_0,r)=0.
\end{equation*}
\end{definition}
Note that
\begin{equation*}
\big\|f\big\|_{{M}^{p,\lambda}}=\sup_{x_0\in \mathbb R^n,r>0}\mathcal{M}_{p,\lambda}(f;x_0,r).
\end{equation*}
The vanishing Morrey space $VM^{p,\lambda}(\mathbb R^n)$ is a Banach space with respect to the norm
\begin{equation*}
\big\|f\big\|_{VM^{p,\lambda}}=\big\|f\big\|_{{M}^{p,\lambda}}=\sup_{x_0\in \mathbb R^n,r>0}\mathcal{M}_{p,\lambda}(f;x_0,r).
\end{equation*}
The vanishing Morrey space $VM^{p,\lambda}(\mathbb R^n)$ was introduced by Vitanza in \cite{vitanza}, where applications to PDEs were considered. Let $1\leq p<\infty$ and $0<\lambda<n$. Then we have
\begin{equation*}
L^{\infty}\cap M^{p,\lambda}\subset VM^{p,\lambda}.
\end{equation*}
In fact, for any ball $\mathcal{B}=B(x_0,r)$ centered at $x_0$ and with radius $r>0$ and $f\in L^{\infty}(\mathbb R^n)\cap M^{p,\lambda}(\mathbb R^n)$, it is easy to see that
\begin{equation*}
\begin{split}
\sup_{x_0\in \mathbb R^n}\mathcal{M}_{p,\lambda}(f;x_0,r)
&\leq\|f\|_{L^{\infty}}\cdot\frac{m(\mathcal{B})^{1/p}}{m(\mathcal{B})^{\lambda/{pn}}}\\
&\lesssim\|f\|_{L^{\infty}}\cdot r^{{(n-\lambda)}/p}.
\end{split}
\end{equation*}
Observe that $n-\lambda>0$. Hence,
\begin{equation*}
\lim_{r\to 0^{+}}\sup_{x_0\in \mathbb R^n}\mathcal{M}_{p,\lambda}(f;x_0,r)=0.
\end{equation*}
This gives $f\in VM^{p,\lambda}(\mathbb R^n)$. Moreover, it can be shown that $VM^{p,\lambda}(\mathbb R^n)$ is a closed subspace of $M^{p,\lambda}(\mathbb R^n)$, we conclude that
\begin{equation*}
\overline{L^{\infty}\cap M^{p,\lambda}}\subset VM^{p,\lambda}.
\end{equation*}
One can see \cite{ak}, \cite{ragusa} and \cite{Samko} for further details and examples of vanishing Morrey spaces.
\begin{example}
Some examples of BMO and VMO functions are listed below:
\begin{enumerate}
	\item Let $2\leq n\in \mathbb{N}$. Define the homogeneous Sobolev space $\dot{W}^{1,n}(\mathbb R^n)$ as follows.
\begin{equation*}
\dot{W}^{1,n}(\mathbb R^n):=\bigg\{f\in \mathcal{C}^{1}(\mathbb R^n):\|f\|_{\dot{W}^{1,n}}=\bigg(\int_{\mathbb R^n}|\nabla f(x)|^ndx\bigg)^{1/n}<+\infty\bigg\}.
\end{equation*}
Here and in what follows, $\nabla$ denotes the standard gradient operator
\begin{equation*}
\Big(\frac{\partial}{\partial x_1},\frac{\partial}{\partial x_2},\dots,\frac{\partial}{\partial x_n}\Big)
\end{equation*}
when it acts on the function $f(x)$ for $x=(x_1,x_2,\dots,x_n)$. The well-known Poincar\'{e} embedding theorem states that $\dot{W}^{1,n}(\mathbb R^n)$ is a proper subspace of $\mathrm{BMO}(\mathbb R^n)$(see \cite[Theorem 4.1]{xiao1}, \cite[Theorem 1.4]{xiao2} for instance). Indeed, by Poincar\'{e}'s inequality
\begin{equation*}
\frac{1}{m(\mathcal{B})}\int_{\mathcal{B}}|f(x)-f_{\mathcal{B}}|\,dx\leq C(n)\bigg(\int_{\mathcal{B}}|\nabla f(x)|^ndx\bigg)^{1/n},
\end{equation*}
we can deduce from this that
\begin{equation*}
\dot{W}^{1,n}(\mathbb R^n)\subset\mathrm{BMO}(\mathbb R^n),
\end{equation*}
and the Poincar\'{e} inclusion is strict.
    \item Let $2\leq n\in \mathbb{N}$. The following conformal invariant Sobolev space, denoted by $\mathrm{CIS}(\mathbb R^n)$, was introduced by Xiao in \cite{xiao1}.
\begin{equation*}
\mathrm{CIS}(\mathbb R^n):=\bigg\{f\in\mathcal{C}^{1}(\mathbb R^n):\|f\|_{\mathrm{CIS}}<+\infty\bigg\},
\end{equation*}
where
\begin{equation*}
\|f\|_{\mathrm{CIS}}:=\sup_{\mathcal{B}\subset\mathbb R^n}\bigg(|\mathcal{B}|^{\frac{2-n}{n}}\int_{\mathcal{B}}|\nabla f(x)|^2dx\bigg)^{1/2}.
\end{equation*}
In other words, the conformal invariant space $\mathrm{CIS}(\mathbb R^n)$ is defined as the set of all functions $f\in\mathcal{C}^{1}(\mathbb R^n)$ with $|\nabla f(x)|\in M^{2,n-2}(\mathbb R^n)$, which turns out to be a proper subspace of $\mathrm{BMO}(\mathbb R^n)$. If $n=2$, then we have
\begin{equation*}
\mathrm{CIS}(\mathbb R^n)=\dot{W}^{1,2}(\mathbb R^n)(|\nabla f|\in L^2(\mathbb R^n)).
\end{equation*}
If $n\geq3$, then by using H\"{o}lder's inequality, we deduce that for an arbitrary fixed ball $\mathcal{B}\subset\mathbb R^n$,
\begin{equation*}
\begin{split}
\bigg(\int_{\mathcal{B}}|\nabla f(x)|^2dx\bigg)^{1/2}
&\leq\bigg(\int_{\mathcal{B}}|\nabla f(x)|^ndx\bigg)^{1/n}m(\mathcal{B})^{1/2-1/n}\\
&\leq\bigg(\int_{\mathbb R^n}|\nabla f(x)|^ndx\bigg)^{1/n}m(\mathcal{B})^{1/2-1/n},
\end{split}
\end{equation*}
which implies that
\begin{equation*}
m(\mathcal{B})^{1/n-1/2}\bigg(\int_{\mathcal{B}}|\nabla f(x)|^2dx\bigg)^{1/2}\leq \bigg(\int_{\mathbb R^n}|\nabla f(x)|^ndx\bigg)^{1/n}
\end{equation*}
holds for any ball $\mathcal{B}\subset\mathbb R^n$ and $f\in \mathcal{C}^{1}(\mathbb R^n)$, and hence
\begin{equation*}
\dot{W}^{1,n}(\mathbb R^n)\subset\mathrm{CIS}(\mathbb R^n)\quad \&\quad \|f\|_{\mathrm{CIS}}\leq\|f\|_{\dot{W}^{1,n}}=\|\nabla f\|_{L^n}.
\end{equation*}
Moreover, the left inclusion is strict (see \cite[Theorem 4.1]{xiao1} and \cite[Theorem 1.4]{xiao2}). Similarly, for any $1\leq p\leq n$, we are motivated by \cite{xiao1} and \cite{xiao2} to define
\begin{equation*}
\mathrm{CIS}^p(\mathbb R^n):=\bigg\{f\in\mathcal{C}^{1}(\mathbb R^n):\|f\|_{\mathrm{CIS}^p}<+\infty\bigg\},
\end{equation*}
where
\begin{equation*}
\big\|f\big\|_{\mathrm{CIS}^p}:=\big\|\nabla f\big\|_{M^{p,n-p}}
=\sup_{\mathcal{B}\subset\mathbb R^n}\bigg(m(\mathcal{B})^{\frac{p-n}{n}}\int_{\mathcal{B}}|\nabla f(x)|^pdx\bigg)^{1/p}.
\end{equation*}
By using H\"{o}lder's inequality once again, for each fixed ball $\mathcal{B}\subset\mathbb R^n$, we can see that
\begin{equation*}
\bigg(\frac{1}{m(\mathcal{B})}\int_{\mathcal{B}}|\nabla f(x)|^2dx\bigg)^{1/2}\leq\bigg(\frac{1}{m(\mathcal{B})}\int_{\mathcal{B}}|\nabla f(x)|^pdx\bigg)^{1/p}
\end{equation*}
whenever $2\leq p\leq n$. Then we have
\begin{equation*}
\big\|f\big\|_{\mathrm{CIS}}\leq \big\|f\big\|_{\mathrm{CIS}^p}\quad \&\quad \mathrm{CIS}^p(\mathbb R^n)\subset\mathrm{CIS}(\mathbb R^n)\subset \mathrm{BMO}(\mathbb R^n).
\end{equation*}
In particular, when $p=n$, we immediately obtain
\begin{equation*}
\mathrm{CIS}^p(\mathbb R^n)=\dot{W}^{1,n}(\mathbb R^n)\subset\mathrm{BMO}(\mathbb R^n)\quad\&\quad \big\|f\big\|_{\mathrm{CIS}^p}=\big\|\nabla f\big\|_{L^{n}}=\big\|f\big\|_{\dot{W}^{1,n}}.
\end{equation*}
We now deal with the case of $1\leq p<2$. When $p=1$, by using Poincar\'{e}'s inequality, we know that for any ball $\mathcal{B}\subset\mathbb R^n$,
\begin{equation}\label{poin}
\int_{\mathcal{B}}|f(x)-f_{\mathcal{B}}|\,dx\leq C(n)\cdot m(\mathcal{B})^{1/n}\int_{\mathcal{B}}|\nabla f(x)|\,dx
\end{equation}
holds(see \cite[p.285]{mus}, for example). From this, we can deduce that
\begin{equation*}
\mathrm{CIS}^1(\mathbb R^n)\subset\mathrm{BMO}(\mathbb R^n)\quad\&\quad
\big\|f\big\|_{\mathrm{BMO}}\leq C(n)\big\|f\big\|_{\mathrm{CIS}^1}=C(n)\big\|\nabla f\big\|_{M^{1,n-1}},
\end{equation*}
where $C(n)$ is a positive constant depending only on the dimension $n$. Finally, when $1<p<2$, by following the steps as in the case $2<p<n$, one has
\begin{equation*}
\mathrm{CIS}(\mathbb R^n)\subset\mathrm{CIS}^p(\mathbb R^n)\subset\mathrm{BMO}(\mathbb R^n).
\end{equation*}
We summarize our discussion up to this point: functions that have gradient in $L^n(\mathbb R^n)$ are in the space $\mathrm{BMO}(\mathbb R^n)$, and functions that have gradient in $M^{p,n-p}(\mathbb R^n)$ are also in the space $\mathrm{BMO}(\mathbb R^n)$, where $1\leq p<n$. Furthermore, in view of \eqref{poin},   functions that have gradient in $VM^{p,n-p}(\mathbb R^n)$ are in the space $\mathrm{VMO}(\mathbb R^n)$.
	\item Let $P(x)$ be any polynomial function on $\mathbb R^n$. Then the function $\log|P(x)|$ belongs to the space $\mathrm{BMO}(\mathbb R^n)$. It is obvious that bounded uniformly continuous functions belong to the space $\mathrm{VMO}(\mathbb R^n)$. An useful example of discontinuous VMO function is $f(x):=|\log|x||^{\delta}$(unbounded), $0<\delta<1$. It should be pointed out that $\mathrm{VMO}(\mathbb R^n)$ is a closed subspace of $\mathrm{BMO}(\mathbb R^n)$. $\mathrm{VMO}(\mathbb R^n)$ can also be characterized as the subspace of $\mathrm{BMO}(\mathbb R^n)$ on which translation is continuous. One can see \cite{Dafni1} and \cite{Dafni2} for the recent development and applications of VMO space.
\end{enumerate}
\end{example}
A well known result by Adams states that for $1<p<q<\infty$, the classical fractional integral $\mathcal{I}_{\alpha}$ of order $\alpha$ is bounded from $M^{p,\lambda}(\mathbb R^n)$ to $M^{q,\lambda}(\mathbb R^n)$, where $0<\alpha<n$, $0<\lambda<n-\alpha p$ and $1/q=1/p-\alpha/{(n-\lambda)}$. See \cite{adams} and \cite{adams1}. This result extends the classical Sobolev embedding theorem to Morrey spaces. Moreover, in the limiting case $\lambda=n-\alpha p$, it was shown by Adams that the classical fractional integral operator $\mathcal{I}_{\alpha}$ is bounded from $M^{p,\lambda}(\mathbb R^n)$ into $\mathrm{BMO}(\mathbb R^n)$. See also \cite{adams} and \cite{adams1}. As a direct consequence of \eqref{dominate1}, we have the following estimate for the operator $\mathcal{L}^{-\alpha/2}$ on Morrey spaces.
\begin{theorem}
Let $0<\alpha<n$ and $1<p<{(n-\lambda)}/\alpha$. If $0<\lambda<n-\alpha p$ and $1/q=1/p-\alpha/{(n-\lambda)}$, then the generalized fractional integral operator $\mathcal L^{-\alpha/2}$ is bounded from $M^{p,\lambda}(\mathbb R^n)$ to $M^{q,\lambda}(\mathbb R^n)$.
\end{theorem}
Then, a natural question occurs to us: what happens for the operator $\mathcal{L}^{-\alpha/2}$ if $\lambda=n-\alpha p$?The main purpose of this paper is to prove that the operator $\mathcal{L}^{-\alpha/2}$ is bounded from ${M}^{p,\lambda}(\mathbb R^n)$ into $\mathrm{BMO}_{\mathcal{L}}(\mathbb R^n)$, where $0<\alpha<n$ and $1\leq p<n/{\alpha}$. Moreover, we also show that the operator $\mathcal{L}^{-\alpha/2}$ is bounded from $V{M}^{p,\lambda}(\mathbb R^n)$ into $\mathrm{VMO}_{\mathcal{L}}(\mathbb R^n)$, where $\mathrm{BMO}_{\mathcal{L}}(\mathbb R^n)$ and $\mathrm{VMO}_{\mathcal{L}}(\mathbb R^n)$ are the spaces of bounded mean oscillation and vanishing mean oscillation associated with the operator $\mathcal{L}$, respectively (see Definitions \ref{newbmo} and \ref{newvmo} below). The main results are given in Theorems \ref{thm1} and \ref{thm2}. The proof is based on the pointwise estimate \eqref{kernelk} and the kernel estimate for the
difference operator $\mathcal{L}^{-\alpha/2}-e^{-t\mathcal L}\mathcal{L}^{-\alpha/2}$ with $0<\alpha<n$ (see Lemma \ref{wanglemma2} below), as well as some properties of ${M}^{p,\lambda}(\mathbb R^n)$ and $V{M}^{p,\lambda}(\mathbb R^n)$. In particular, when $\mathcal{L}=-\Delta$ is the Laplacian operator on $\mathbb R^n$, we can get the corresponding results for $\mathcal{I}_{\alpha}$ with $0<\alpha<n$. To the best of our knowledge, such a ($V{M}^{p,\lambda}$,VMO) result was first proved by Rafeiro--Samko in \cite[Theorem 4]{Rafeiro}. As a corollary, we obtain that for all $0<\alpha<n$ and $p=n/{\alpha}$, the operator $\mathcal{L}^{-\alpha/2}$ is bounded from $L^{p,\infty}(\mathbb R^n)$ into $\mathrm{BMO}_{\mathcal{L}}(\mathbb R^n)$ since $L^{p,\infty}(\mathbb R^n)$ is continuously embedded into ${M}^{q,\lambda}(\mathbb R^n)$ provided that $1\leq q<p$ and $\lambda=n(1-q/p)$. In particular, when $\mathcal{L}=-\Delta$ is the Laplacian operator on $\mathbb R^n$, this ($L^{p,\infty}$,BMO) result for classical fractional integral $\mathcal{I}_{\alpha}$ was previously known, see \cite[p.164]{stein}.

\section{Definitions, notations and assumptions}
\label{sec2}
In this section, we will recall the function spaces of bounded mean oscillation and vanishing mean oscillation associated with the operator $\mathcal{L}$, and state a lemma, which plays an important role in the proofs of our main theorems.
\subsection{Definitions of the spaces $\mathrm{BMO}_{\mathcal{L}}(\mathbb R^n)$ and $\mathrm{VMO}_{\mathcal{L}}(\mathbb R^n)$}
We first give some preliminary definitions of ``generalized approximation to the identity" and the related sharp maximal function, which were introduced by Martell in \cite{martell}.

A family of operators $\{\mathbf{A}_t:t>0\}$ is said to be ``generalized approximation to the identity", if for every $t>0$, $\mathbf{A}_t$ is represented by the kernel $p_t(x,y)$, which is a measurable function defined on $\mathbb R^n\times \mathbb R^n$ and satisfies an upper bound
\begin{equation*}
\big|p_t(x,y)\big|\leq h_t(x,y):=t^{-n/2}g\bigg(\frac{|x-y|}{\sqrt{t\,}}\bigg),
\end{equation*}
for all $x,y\in\mathbb R^n$ and all $t>0$. Here $g$ is a positive, bounded and decreasing function satisfying
\begin{equation}\label{15}
\lim_{r\to\infty}r^{n+\varepsilon}g(r)=0
\end{equation}
for some $\varepsilon>0$.

Let $\varepsilon$ be as in \eqref{15} and let $0<\varrho<\varepsilon$. A locally integrable function $f$ is said to be a function of $\varrho$-type, if it satisfies
\begin{equation}\label{16}
\bigg(\int_{\mathbb R^n}\frac{|f(x)|^2}{(1+|x|)^{n+\varrho}}dx\bigg)^{1/2}\leq C<+\infty,
\end{equation}
and we denote by $\mathcal{M}_{\varrho}$ the collection of all functions of $\varrho$-type. The smallest bound $C>0$ satisfying the condition \eqref{16} is then taken to be the norm of $f$ in $\mathcal{M}_{\varrho}$, and is denoted by $\|f\|_{\mathcal{M}_{\varrho}}$. It is easy to see that $\mathcal{M}_{\varrho}$ is a Banach function space under the norm $\|f\|_{\mathcal{M}_{\varrho}}$. We set
\begin{equation*}
\mathcal{M}:=\bigcup_{\varrho:0<\varrho<\varepsilon}\mathcal{M}_{\varrho}.
\end{equation*}
For any given $\{\mathbf{A}_t:t>0\}$ and $f\in \mathcal{M}$, the sharp maximal function $M^{\#}_{\mathbf{A}}(f)$ associated with ``generalized approximation to the identity" is defined as follows:
\begin{equation*}
M^{\#}_{\mathbf{A}}(f)(x):=\sup_{x\in \mathcal{B}}\frac{1}{m(\mathcal{B})}
\int_{\mathcal{B}}\big|f(y)-\mathbf{A}_{t_{\mathcal{B}}}f(y)\big|\,dy,
\end{equation*}
where $t_{\mathcal{B}}=r_{\mathcal{B}}^2$, $r_{\mathcal{B}}$ is the radius of the ball $\mathcal{B}$,
\begin{equation*}
\mathbf{A}_{t_{\mathcal{B}}}f(y):=\int_{\mathbb R^n}p_{t_{\mathcal{B}}}(y,z)f(z)\,dz,
\end{equation*}
and the supremum is taken over all balls $\mathcal{B}$ containing the point $x$. The sharp maximal function $M^{\#}_{\mathbf{A}}$ was first introduced and studied by Martell in \cite{martell}. We remark that our analytic semigroup $\big\{e^{-t\mathcal L}:t>0\big\}$ is ``generalized approximation to the identity". By virtue of the Gaussian upper bound \eqref{G} we can extend the action of the semigroup $\big\{e^{-t\mathcal L}:t>0\big\}$ to the space $\mathcal{M}$. Namely, we can define $e^{-t\mathcal L}(f)$ for all $f\in \mathcal{M}$. In particular, we give the sharp maximal function $M^{\#}_{\mathcal{L}}$ associated with $\big\{e^{-t\mathcal L}:t>0\big\}$.
\begin{definition}\label{martel}
For any $f\in \mathcal{M}$, the sharp maximal function $M^{\#}_{\mathcal{L}}(f)$ is defined by
\begin{equation*}
M^{\#}_{\mathcal{L}}(f)(x):=\sup_{x\in \mathcal{B}}\frac{1}{m(\mathcal{B})}
\int_{\mathcal{B}}\big|f(y)-e^{-t_{\mathcal{B}}\mathcal{L}}f(y)\big|\,dy,
\end{equation*}
where $t_{\mathcal{B}}=r_{\mathcal{B}}^2$, $r_{\mathcal{B}}$ is the radius of the ball $\mathcal{B}$ and the supremum is taken over all balls $\mathcal{B}$ in $\mathbb R^n$.
\end{definition}
\begin{definition}\label{newbmo}
Let $f\in \mathcal{M}$. We say that a function $f$ is in the space $\mathrm{BMO}_{\mathcal{L}}(\mathbb R^n)$ associated with the operator $\mathcal{L}$, if the sharp maximal function $M^{\#}_{\mathcal{L}}(f)\in L^{\infty}(\mathbb R^n)$, and we define
\begin{equation}\label{newbmonorm}
\big\|f\big\|_{\mathrm{BMO}_{\mathcal{L}}}:=\big\|M^{\#}_{\mathcal{L}}(f)\big\|_{L^{\infty}}.
\end{equation}
\end{definition}
The new $\mathrm{BMO}_{\mathcal{L}}(\mathbb R^n)$ space associated with the operator $\mathcal{L}$ was first introduced and studied by Duong and Yan in \cite{duong3}(see also \cite{deng2,deng,duong2}). The idea is that the quantity $e^{-t\mathcal{L}}(f)$ can be viewed as an average version of $f$(at the scale $t$) and the quantity $e^{-t_{\mathcal{B}}\mathcal{L}}f(x)$ can be used to replace the mean value $f_{\mathcal{B}}$ in the definition of the classical BMO space. Here $t_{\mathcal{B}}$ is equal to the square of the radius of $\mathcal{B}$.
\begin{remark}
\begin{enumerate}
Let us now list some important properties of the space $\mathrm{BMO}_{\mathcal{L}}(\mathbb R^n)$.
\item Notice that $\mathrm{BMO}_{\mathcal{L}}(\mathbb R^n)$ is a semi-normed vector space, with the semi-norm vanishing on the kernel space $\mathrm{Ker}_{\mathcal{L}}$, which is defined by
\begin{equation*}
\mathrm{Ker}_{\mathcal{L}}:=\Big\{f\in \mathcal{M}:e^{-t\mathcal{L}}(f)=f,~~\mbox{for all}~~ t>0\Big\}.
\end{equation*}
The class of $\mathrm{BMO}_{\mathcal{L}}$ functions (modulo $\mathrm{Ker}_{\mathcal{L}}$) is a Banach function space with the norm $\|\cdot\|_{\mathrm{BMO}_{\mathcal{L}}}$ defined as in \eqref{newbmonorm}. For a discussion of the dimensions of $\mathrm{Ker}_{\mathcal{L}}$, when $\mathcal{L}$ is a second-order elliptic operator of divergence form or a Schr\"{o}dinger operator, see \cite{duong2} and \cite{duong3}. Let $\mathcal{B}=B(x_0,r)$ be a ball in $\mathbb R^n$ centered at $x_0$ with radius $r>0$. When $f\in \mathcal{M}$, $M^{\#}_{\mathcal{L}}(f)\in L^{\infty}(\mathbb R^n)$ is equivalent to
\begin{equation*}
\sup_{x_0\in\mathbb R^n,r>0}\frac{1}{m(\mathcal{B})}\int_{\mathcal{B}}\big|f(y)-e^{-t_{\mathcal{B}}\mathcal{L}}f(y)\big|\,dy<+\infty.
\end{equation*}
\item A natural question arising from Definition \ref{newbmo} is to compare the classical BMO space with the $\mathrm{BMO}_{\mathcal{L}}$ space associated with the operator $\mathcal{L}$. Denote by $e^{t\Delta}$ the heat semigroup on $\mathbb R^n$. It can be shown that the classical BMO space (modulo all constant functions) and the $\mathrm{BMO}_{\Delta}$ space (modulo $\mathrm{Ker}_{\Delta}$) coincide, and their norms are equivalent, see \cite{duong2} and \cite{duong3}. As in the classical case, a variant of the John--Nirenberg inequality holds for $\mathrm{BMO}_{\mathcal{L}}$ functions. This result and H\"{o}lder's inequality imply that for any $1<p<\infty$,
\begin{equation*}
\big\|f\big\|_{\mathrm{BMO}_{\mathcal{L}}}
\leq\bigg(\frac{1}{m(\mathcal{B})}\int_{\mathcal{B}}\big|f(y)-e^{-t_{\mathcal{B}}\mathcal{L}}f(y)\big|^pdy\bigg)^{1/p}
\leq C\big\|f\big\|_{\mathrm{BMO}_{\mathcal{L}}}
\end{equation*}
with $C>0$ depending only on the dimension $n$ and $p$. Let $\mathrm{BMO}^p_{\mathcal{L}}(\mathbb R^n)$ be the set of all locally integrable functions $f\in \mathcal{M}$ such that
\begin{equation*}
\big\|f\big\|_{\mathrm{BMO}^p_{\mathcal{L}}}:=
\sup_{\mathcal{B}\subset\mathbb R^n}\bigg(\frac{1}{m(\mathcal{B})}
\int_{\mathcal{B}}\big|f(y)-e^{-t_{\mathcal{B}}\mathcal{L}}f(y)\big|^pdy\bigg)^{1/p}<+\infty.
\end{equation*}
Then for all $1\leq p<\infty$, the spaces $\mathrm{BMO}^p_{\mathcal{L}}(\mathbb R^n)$ coincide, and the norms $\|\cdot\|_{\mathrm{BMO}^p_{\mathcal{L}}}$ are equivalent with respect to different values of $p$. See \cite{deng2,duong2,duong3} for more details.
\item Assume that for every $t>0$, the equation
\begin{equation*}
e^{-t\mathcal{L}}(\mathbf{1})(x)=\mathbf{1}
\end{equation*}
holds for almost everywhere $x\in\mathbb R^n$, that is,
\begin{equation*}
\int_{\mathbb R^n}\mathcal{P}_t(x,y)\,dy=1
\end{equation*}
for almost all $x\in\mathbb R^n$. Then we have $\mathrm{BMO}(\mathbb R^n)\subset\mathrm{BMO}_{\mathcal{L}}(\mathbb R^n)$, and there exists a positive constant $C>0$ such that
\begin{equation}\label{18}
\big\|f\big\|_{\mathrm{BMO}_{\mathcal{L}}}\leq C\big\|f\big\|_{\mathrm{BMO}}.
\end{equation}
However, the converse inequality does not hold in general. See Martell \cite[Proposition 3.1]{martell} and Deng--Duong--Sikora--Yan \cite[Proposition 2.3]{deng}. As pointed out in \cite{deng}, the condition $e^{-t\mathcal{L}}(\mathbf{1})(x)=\mathbf{1}$ is also necessary for \eqref{18}. Indeed, it follows from \eqref{18} that $\big\|\mathbf{1}\big\|_{\mathrm{BMO}_{\mathcal{L}}}=0$. This in turn implies that for all $t>0$, $e^{-t\mathcal{L}}(\mathbf{1})(x)=\mathbf{1}$ holds for almost everywhere $x\in\mathbb R^n$. Under the additional condition that the kernel $\mathcal{P}_t(x,y)$ of $e^{-t\mathcal L}$ has sufficient regularities on space variables $x,y$ and
\begin{equation*}
e^{-t\mathcal{L}}(\mathbf{1})(x)=e^{-t\mathcal{L}^{*}}(\mathbf{1})(x)=\mathbf{1}
\end{equation*}
holds for almost all $x\in\mathbb R^n$, it can be verified that the classical space $\mathrm{BMO}(\mathbb R^n)$ and $\mathrm{BMO}_{\mathcal{L}}(\mathbb R^n)$ coincide, and their norms are equivalent, where $\mathcal{L}^{*}$ denotes the adjoint operator of $\mathcal{L}$. For further details about the applications of the space $\mathrm{BMO}_{\mathcal{L}}(\mathbb R^n)$ associated with the operator $\mathcal{L}$, we refer the reader to \cite{deng2,deng,duong2,duong3} and the references therein.
\end{enumerate}
\end{remark}
Let us now introduce the new space $\mathrm{VMO}_{\mathcal{L}}(\mathbb R^n)$ associated with the operator $\mathcal{L}$, which generalizes the classical VMO space.
\begin{definition}\label{newvmo}
Let $f\in \mathcal{M}$. We say that a function $f\in \mathrm{BMO}_{\mathcal{L}}(\mathbb R^n)$ is in the space $\mathrm{VMO}_{\mathcal{L}}(\mathbb R^n)$ associated with the operator $\mathcal{L}$, if it satisfies the limiting condition
\begin{equation*}
\lim_{r_{\mathcal{B}}\to 0^{+}}\eta(f;r_{\mathcal{B}})=0,
\end{equation*}
where
\begin{equation}\label{modulus}
\eta(f;r_{\mathcal{B}}):=\sup_{x_0\in\mathbb R^n}\frac{1}{m(\mathcal{B})}
\int_{\mathcal{B}}\big|f(y)-e^{-t_{\mathcal{B}}\mathcal{L}}f(y)\big|\,dy,
\end{equation}
$t_{\mathcal{B}}=r_{\mathcal{B}}^2$ and $r_{\mathcal{B}}>0$ is the radius of the ball $\mathcal{B}=B(x_0,r_{\mathcal{B}})$. Moreover, we endow $\mathrm{VMO}_{\mathcal{L}}(\mathbb R^n)$ with the norm of $\mathrm{BMO}_{\mathcal{L}}(\mathbb R^n)$.
\end{definition}
The space $\mathrm{VMO}_{\mathcal{L}}(\mathbb R^n)$ was introduced and developed by Deng--Duong--Song--Tan--Yan in \cite{d}. We can also show that when $\mathcal{L}=-\Delta$ is the Laplacian operator on $\mathbb R^n$, then the corresponding space $\mathrm{VMO}_{\mathcal{L}}(\mathbb R^n)$ coincides with the classical space $\mathrm{VMO}(\mathbb R^n)$, and their norms are equivalent. See the same paper \cite{d}.

\subsection{A key lemma}
To study the operator $\mathcal{L}^{-\alpha/2}$, we also need the following key lemma, which gives the kernel estimate of the difference operator $\mathcal{L}^{-\alpha/2}-e^{-t\mathcal L}\mathcal{L}^{-\alpha/2}$ for any $t>0$ and $0<\alpha<n$. Observe that
\begin{equation*}
I-e^{-t\mathcal L}=\int_0^{t}Le^{-sL}ds.
\end{equation*}
Hence, by using \eqref{gefrac}, we see that for any $t>0$ and $0<\alpha<n$,
\begin{equation*}
\begin{split}
\mathcal{L}^{-\alpha/2}-e^{-t\mathcal L}\mathcal{L}^{-\alpha/2}&=(I-e^{-t\mathcal L})\mathcal{L}^{-\alpha/2}\\
&=\frac{1}{\Gamma(\alpha/2)}\int_0^{t}\int_0^{+\infty}Le^{-(s+r)L}\cdot r^{\alpha/2-1}drds\\
&=\frac{1}{\Gamma(\alpha/2)}\int_0^{t}\int_0^{+\infty}(s+r)Le^{-(s+r)L}\cdot\frac{r^{\alpha/2-1}}{s+r}drds.
\end{split}
\end{equation*}
It can be proved that the kernel of the operator $vLe^{-vL}$ also satisfies Gaussian upper bound \eqref{G}(see, for example, \cite{deng} and \cite{d}). Hence, the operator $(I-e^{-t\mathcal L})\mathcal{L}^{-\alpha/2}$ has an associated kernel $\widetilde{K}_{\alpha,t}(x,y)$ which satisfies
\begin{equation*}
\widetilde{K}_{\alpha,t}(x,y)\leq C\int_0^{t}\int_0^{+\infty}\frac{1}{(s+r)^{n/2}}e^{-A\frac{|x-y|^2}{s+r}}\cdot\frac{r^{\alpha/2-1}}{s+r}drds.
\end{equation*}
Then we have the following estimate for $\widetilde{K}_{\alpha,t}(x,y)$.
\begin{lemma}\label{wanglemma2}
Assume that the semigroup $\big\{e^{-t\mathcal L}:t>0\big\}$ has a kernel $\mathcal{P}_t(x,y)$ satisfying the Gaussian upper bound \eqref{G}.
Then for any $0<\alpha<n$, the difference operator $(I-e^{-t\mathcal L})\mathcal{L}^{-\alpha/2}:=\mathcal{L}^{-\alpha/2}-e^{-t\mathcal L}\mathcal{L}^{-\alpha/2}$ has an associated kernel $\widetilde{K}_{\alpha,t}(x,y)$ which satisfies
\begin{equation*}
\widetilde{K}_{\alpha,t}(x,y)\leq\frac{C}{|x-y|^{n-\alpha}}\cdot\frac{t}{|x-y|^2}.
\end{equation*}
Here the constant $C$ is independent of $x,y\in\mathbb R^n$ and $t\in(0,+\infty)$.
\end{lemma}

For the proof of this lemma, see Duong--Yan
\cite[Lemma 3.1]{duong1} for $0<\alpha<1$ and Deng--Duong--Sikora--Yan \cite[Lemma 5.3]{deng} for $0<\alpha<n$.

Throughout this paper the letter $C$ stands for a positive constant not necessarily the same at each occurrence. The notation $\mathbf{X}\approx\mathbf{Y}$ means that $C_1\mathbf{Y}\leq \mathbf{X}\leq C_2\mathbf{Y}$ with some positive constants $C_1$ and $C_2$. For given $\mathcal{B}=B(x_0,r_{B})$ and $\lambda>0$, we write $\lambda \mathcal{B}$ for the $\lambda$-dilate ball, which is the ball with the same center $x_0$ and with radius $\lambda r_{B}$. For a measurable set $E$ in $\mathbb R^n$, $m(E)$ denotes the Lebesgue measure of the set $E$ and $\chi_{E}$ denotes the characteristic function of the set $E$.

\section{Main results}
Let $0<\alpha<n$ and $\mathcal L^{-\alpha/2}$ be the generalized fractional integral defined in \eqref{gefrac}. We will prove the boundedness of $\mathcal L^{-\alpha/2}$ in the limiting Sobolev case $\alpha p=n-\lambda$ with $1\leq p<n/{\alpha}$. It is shown that for all $0<\alpha<n$, the operator $\mathcal L^{-\alpha/2}$ is bounded from ${M}^{p,\lambda}(\mathbb R^n)$ into $\mathrm{BMO}_{\mathcal{L}}(\mathbb R^n)$, and also bounded from $V{M}^{p,\lambda}(\mathbb R^n)$ into $\mathrm{VMO}_{\mathcal{L}}(\mathbb R^n)$. The main results in the paper are formulated as follows.
\begin{theorem}\label{thm1}
Let $0<\alpha< n$ and $1\leq p<n/{\alpha}$. Suppose that $\lambda=n-\alpha p$. Then for any $f\in{M}^{p,\lambda}(\mathbb R^n)$, there exists a positive constant $C>0$ independent of $f$ such that
\begin{equation*}
\big\|\mathcal L^{-\alpha/2}(f)\big\|_{\mathrm{BMO}_{\mathcal{L}}}
\leq C\big\|f\big\|_{{M}^{p,\lambda}}.
\end{equation*}
\end{theorem}

\begin{theorem}\label{thm2}
Let $0<\alpha<n$ and $1\leq p<n/{\alpha}$. Suppose that $\lambda=n-\alpha p$. Then for any $f\in V{M}^{p,\lambda}(\mathbb R^n)$, there exists a positive constant $C>0$ independent of $f$ such that
\begin{equation*}
\big\|\mathcal L^{-\alpha/2}(f)\big\|_{\mathrm{VMO}_{\mathcal{L}}}
\leq C\big\|f\big\|_{V{M}^{p,\lambda}},
\end{equation*}
and
\begin{equation*}
\lim_{r_{\mathcal{B}}\to 0^{+}}\eta(\mathcal L^{-\alpha/2}(f);r_{\mathcal{B}})=0.
\end{equation*}
Here $\eta(\mathcal L^{-\alpha/2}(f);r_{\mathcal{B}})$ is defined as in \eqref{modulus}.
\end{theorem}

In particular, if $\mathcal{L}=-\Delta$ is the Laplacian on $\mathbb R^n$, then we obtain that the classical fractional integral $\mathcal{I}_{\alpha}$ is bounded from ${M}^{p,\lambda}(\mathbb R^n)$ into $\mathrm{BMO}(\mathbb R^n)$. This well-known result was first proved by Adams (see \cite{adams} and \cite{adams1}). We also obtain that $\mathcal{I}_{\alpha}$ is bounded from $V{M}^{p,\lambda}(\mathbb R^n)$ into $\mathrm{VMO}(\mathbb R^n)$. This interesting result was first established by Rafeiro and Samko, to the author's best knowledge (see \cite{Rafeiro}).

Moreover, it was shown that the inclusion relation
\begin{equation*}
L^{p,\infty}(\mathbb R^n)\hookrightarrow M^{q,\lambda}(\mathbb R^n)
\end{equation*}
holds for $1\leq q<p$ and $\lambda=n(1-q/p)$(see \cite[equation (21)]{wang2} and \cite[equation (6.1)]{wang3}). Note that
\begin{equation*}
p=n/{\alpha}\Longrightarrow \lambda=n(1-q/p)=n-\alpha q.
\end{equation*}
Hence, in the special case of $p=n/{\alpha}$ and $0<\alpha<n$, one has
\begin{equation*}
L^{p,\infty}(\mathbb R^n)\hookrightarrow M^{q,\lambda}(\mathbb R^n)
\end{equation*}
with $1\leq q<n/{\alpha}$ and $\lambda=n-\alpha q$. As an immediate consequence of Theorem \ref{thm1}, we have the following estimate for the generalized fractional integral $\mathcal L^{-\alpha/2}$ on weak Lebesgue spaces.
\begin{corollary}\label{thm3}
Let $0<\alpha<n$ and $p=n/{\alpha}$. Then for any $f\in L^{p,\infty}(\mathbb R^n)$, there exists a positive constant $C>0$ independent of $f$ such that
\begin{equation*}
\big\|\mathcal L^{-\alpha/2}(f)\big\|_{\mathrm{BMO}_{\mathcal{L}}}
\leq C\big\|f\big\|_{L^{p,\infty}}.
\end{equation*}
\end{corollary}
In particular, if $\mathcal{L}=-\Delta$ is the Laplacian on $\mathbb R^n$, then we obtain that the classical fractional integral $\mathcal{I}_{\alpha}$ is bounded from $L^{p,\infty}(\mathbb R^n)$ into $\mathrm{BMO}(\mathbb R^n)$. This classical result can be found in  \cite[p.164]{stein}.

\section{Proofs of Theorems $\ref{thm1}$ and $\ref{thm2}$}
In this section, we will give the proofs of our main theorems. We first note that if $\in{M}^{p,\lambda}(\mathbb R^n)$ with $1\leq p<n/{\alpha}$, then by H\"{o}lder's inequality, we see that for any ball $\mathcal{B}$ in $\mathbb R^n$,
\begin{equation}\label{oftenused}
\begin{split}
\int_{\mathcal{B}}|f(y)|dy
&\leq\bigg(\int_{\mathcal{B}}|f(y)|^pdy\bigg)^{1/p}m(\mathcal{B})^{1/{p'}}\\
&\leq\big\|f\big\|_{M^{p,\lambda}}m(\mathcal{B})^{\lambda/{pn}}\cdot m(\mathcal{B})^{1/{p'}},
\end{split}
\end{equation}
where $p'=p/{(p-1)}$ denotes the H\"{o}lder conjugate exponent of $p$, and $p'=\infty$ if $p=1$.

We are now ready to prove Theorem \ref{thm1}.
\begin{proof}[Proof of Theorem $\ref{thm1}$]
Assume that $f\in {M}^{p,\lambda}(\mathbb R^n)$ with $1\leq p<n/{\alpha}$ and $\lambda=n-\alpha p$. By the definition of $\mathrm{BMO}_{\mathcal{L}}(\mathbb R^n)$, it suffices to prove that there exists a constant $C>0$ such that for any ball $\mathcal{B}$ in $\mathbb R^n$,
\begin{equation}\label{main11}
\frac{1}{m(\mathcal{B})}\int_{\mathcal{B}}\big|\mathcal L^{-\alpha/2}f(x)-e^{-t_{\mathcal{B}}\mathcal{L}}
\big(\mathcal L^{-\alpha/2}f\big)(x)\big|dx\leq C\big\|f\big\|_{M^{p,\lambda}}.
\end{equation}
Let $\mathcal{B}=B(x_0,r_{\mathcal{B}})$ be a fixed ball centered at $x_0$ and with radius $r_{\mathcal{B}}$ and $t_{\mathcal{B}}=(r_{\mathcal{B}})^2$. We decompose the function $f(x)$ as follows:
\begin{equation*}
f(x)=f(x)\cdot\chi_{2\mathcal{B}}+f(x)\cdot\chi_{(2\mathcal{B})^{\complement}}:=f_1(x)+f_2(x),
\end{equation*}
where $2\mathcal{B}=B(x_0,2r_{\mathcal{B}})$ and $(2\mathcal{B})^{\complement}=\mathbb R^n\setminus(2\mathcal{B})$. Observe that
\begin{equation*}
\mathcal L^{-\alpha/2}(f)(x)=\mathcal{L}^{-\alpha/2}(f_1)(x)
+\mathcal{L}^{-\alpha/2}(f_2)(x),
\end{equation*}
and
\begin{equation*}
\begin{split}
e^{-t_{\mathcal{B}}\mathcal{L}}\big(\mathcal L^{-\alpha/2}f\big)(x)
&=e^{-t_{\mathcal{B}}\mathcal{L}}\mathcal{L}^{-\alpha/2}(f_1)(x)+
e^{-t_{\mathcal{B}}\mathcal{L}}\mathcal{L}^{-\alpha/2}(f_2)(x).
\end{split}
\end{equation*}
Then we write
\begin{equation*}
\begin{split}
&\frac{1}{m(\mathcal{B})}\int_{\mathcal{B}}\big|\mathcal L^{-\alpha/2}f(x)-e^{-t_{\mathcal{B}}\mathcal{L}}
\big(\mathcal L^{-\alpha/2}f\big)(x)\big|dx\\
&\leq\frac{1}{m(\mathcal{B})}\int_{\mathcal{B}}\big|\mathcal{L}^{-\alpha/2}(f_1)(x)\big|dx
+\frac{1}{m(\mathcal{B})}\int_{\mathcal{B}}\big|e^{-t_{\mathcal{B}}\mathcal{L}}\mathcal{L}^{-\alpha/2}(f_1)(x)\big|dx\\
&+\frac{1}{m(\mathcal{B})}\int_{\mathcal{B}}\big|\mathcal{L}^{-\alpha/2}(f_2)(x)-e^{-t_{\mathcal{B}}\mathcal{L}}\mathcal{L}^{-\alpha/2}(f_2)(x)\big|dx\\
&:=I_1+I_2+I_3.
\end{split}
\end{equation*}

\textbf{Estimation of $I_1$:} Using the previous estimate \eqref{kernelk}, we obtain that for any $x\in \mathcal{B}$,
\begin{equation*}
\begin{split}
\big|\mathcal{L}^{-\alpha/2}(f_1)(x)\big|&\leq\int_{2\mathcal{B}}\big|\mathcal K_{\alpha}(x,y)\big|\cdot|f(y)|dy\\
&\leq C\int_{2\mathcal{B}}\frac{|f(y)|}{|x-y|^{n-\alpha}}dy.
\end{split}
\end{equation*}
This, together with Fubini's theorem, gives us that
\begin{equation*}
\begin{split}
I_1&\leq\frac{C}{m(\mathcal{B})}\int_{\mathcal{B}}\bigg\{\int_{2\mathcal{B}}\frac{|f(y)|}{|x-y|^{n-\alpha}}dy\bigg\}dx\\
&=\frac{C}{m(\mathcal{B})}\int_{2\mathcal{B}}\bigg\{\int_{\mathcal{B}}\frac{1}{|x-y|^{n-\alpha}}dx\bigg\}|f(y)|dy.
\end{split}
\end{equation*}
For any $x\in \mathcal{B}$ and $y\in 2\mathcal{B}$, we have $|x-y|<3r_{\mathcal{B}}$, and hence
\begin{equation}\label{wang1}
\begin{split}
\int_{\mathcal{B}}\frac{1}{|x-y|^{n-\alpha}}dx&=\int_{|x-y|<3r_{\mathcal{B}}}\frac{1}{|x-y|^{n-\alpha}}dx\\
&=\int_0^{3r_{\mathcal{B}}}\int_{\mathbf{S}^{n-1}}\frac{\varrho^{n-1}}{\varrho^{n-\alpha}}\,d\sigma(x')d\varrho\\
&\leq C(r_{\mathcal{B}})^{\alpha}\leq Cm(\mathcal{B})^{\alpha/n}.
\end{split}
\end{equation}
Here $\mathbf{S}^{n-1}:=\{x\in\mathbb R^n:|x|=1\}$ denotes the unit sphere in $\mathbb R^n$ ($n\geq2$) equipped with the normalized Lebesgue measure $d\sigma(x')$ and $x'$ denotes the unit vector in the direction of $x$.
Notice that when $\lambda=n-\alpha p$,
\begin{equation}\label{cal1}
\frac{\lambda}{pn}+\frac{1}{p'}=\frac{\,1\,}{p}-\frac{\alpha}{\,n\,}+\frac{1}{p'}=1-\frac{\alpha}{\,n\,}.
\end{equation}
Consequently, by using the estimates \eqref{wang1} and \eqref{oftenused}, we obtain
\begin{equation*}
\begin{split}
I_1&\leq\frac{C}{m(\mathcal{B})^{1-\alpha/n}}\int_{2\mathcal{B}}|f(y)|dy\\
&\leq\frac{C}{m(\mathcal{B})^{1-\alpha/n}}\big\|f\big\|_{M^{p,\lambda}}m(2\mathcal{B})^{\lambda/{pn}}\cdot m(2\mathcal{B})^{1/{p'}}\\
&\leq C\big\|f\big\|_{M^{p,\lambda}}.
\end{split}
\end{equation*}

\textbf{Estimation of $I_2$:} The kernel of $e^{-t_{\mathcal{B}}\mathcal{L}}$ is denoted by $\mathcal{P}_{t_{\mathcal{B}}}(x,y)$. It is clear that for any $x\in \mathcal{B}$ and $y\in 2\mathcal{B}$, we have
\begin{equation*}
\big|\mathcal{P}_{t_{\mathcal{B}}}(x,y)\big|\leq\frac{C}{(t_{\mathcal{B}})^{n/2}}\leq\frac{C}{m(2\mathcal{B})}
\end{equation*}
by \eqref{G}. Here $t_{\mathcal{B}}=(r_{\mathcal{B}})^2$. Moreover, for any $x\in \mathcal{B}$ and $y\in(2\mathcal{B})^{\complement}$, we have
\begin{equation*}
|x-y|\geq\frac{|y-x_0|}{2}\geq r_{\mathcal{B}}.
\end{equation*}
Consequently, in this case,
\begin{equation*}
\big|\mathcal{P}_{t_{\mathcal{B}}}(x,y)\big|\leq C\cdot\frac{(t_{\mathcal{B}})^{n/2}}{|x-y|^{2n}}
\leq C\cdot\frac{(t_{\mathcal{B}})^{n/2}}{|y-x_0|^{2n}}
\end{equation*}
by using \eqref{G} again.
We then decompose $\mathbb R^n$ into a geometrically increasing sequence of concentric balls, and obtain that for any $x\in \mathcal{B}$,
\begin{equation*}
\begin{split}
&\big|e^{-t_{\mathcal{B}}\mathcal{L}}\mathcal{L}^{-\alpha/2}(f_1)(x)\big|\\
&=\bigg|\int_{\mathbb R^n}\mathcal{P}_{t_{\mathcal{B}}}(x,y)\cdot\mathcal{L}^{-\alpha/2}(f_1)(y)dy\bigg|\\
&\leq \int_{2\mathcal{B}}\big|\mathcal{P}_{t_{\mathcal{B}}}(x,y)\big|\cdot\big|\mathcal{L}^{-\alpha/2}(f_1)(y)\big|dy
+\sum_{k=1}^{\infty}\int_{2^{k+1}\mathcal{B}\setminus 2^k \mathcal{B}}\big|\mathcal{P}_{t_{\mathcal{B}}}(x,y)\big|\cdot\big|\mathcal{L}^{-\alpha/2}(f_1)(y)\big|dy\\
&\leq\frac{C}{m(2\mathcal{B})}\int_{2\mathcal{B}}\big|\mathcal{L}^{-\alpha/2}(f_1)(y)\big|dy
+C\sum_{k=1}^{\infty}\int_{2^{k+1}\mathcal{B}\setminus 2^k \mathcal{B}}\frac{(t_{\mathcal{B}})^{n/2}}{|y-x_0|^{2n}}\cdot\big|\mathcal{L}^{-\alpha/2}(f_1)(y)\big|dy\\
&\leq\frac{C}{m(2\mathcal{B})}\int_{2\mathcal{B}}\big|\mathcal{L}^{-\alpha/2}(f_1)(y)\big|dy
+C\sum_{k=1}^{\infty}\frac{1}{2^{kn}}\cdot\frac{1}{m(2^{k+1}\mathcal{B})}\int_{2^{k+1}\mathcal{B}}\big|\mathcal{L}^{-\alpha/2}(f_1)(y)\big|dy.
\end{split}
\end{equation*}
From the above pointwise estimate, it follows that
\begin{equation*}
\begin{split}
I_2&=\frac{1}{m(\mathcal{B})}\int_{\mathcal{B}}\big|e^{-t_{\mathcal{B}}\mathcal{L}}\mathcal{L}^{-\alpha/2}(f_1)(x)\big|dx\\
&\leq\frac{C}{m(2\mathcal{B})}\int_{2\mathcal{B}}\big|\mathcal{L}^{-\alpha/2}(f_1)(y)\big|dy
+C\sum_{k=1}^{\infty}\frac{1}{2^{kn}}\cdot\frac{1}{m(2^{k+1}\mathcal{B})}\int_{2^{k+1}\mathcal{B}}\big|\mathcal{L}^{-\alpha/2}(f_1)(y)\big|dy\\
&:=I_2'+I_2''.
\end{split}
\end{equation*}
Using the same arguments as in the estimate of $I_1$, we can also deduce that
\begin{equation*}
I_2'\leq C\big\|f\big\|_{M^{p,\lambda}}.
\end{equation*}
On the other hand, for any $k\in \mathbb{N}$, it follows from Fubini's theorem and the kernel estimate \eqref{kernelk} that
\begin{equation*}
\begin{split}
&\frac{1}{m(2^{k+1}\mathcal{B})}\int_{2^{k+1}\mathcal{B}}\big|\mathcal{L}^{-\alpha/2}(f_1)(y)\big|dy\\
&\leq \frac{C}{m(2^{k+1}\mathcal{B})}\int_{2^{k+1}\mathcal{B}}\bigg\{\int_{2\mathcal{B}}\frac{|f(z)|}{|y-z|^{n-\alpha}}dz\bigg\}dy\\
&=\frac{C}{m(2^{k+1}\mathcal{B})}\int_{2\mathcal{B}}\bigg\{\int_{2^{k+1}\mathcal{B}}\frac{1}{|y-z|^{n-\alpha}}dy\bigg\}|f(z)|dz.
\end{split}
\end{equation*}
For any $y\in 2^{k+1}\mathcal{B}$ and $z\in 2\mathcal{B}$, one has $|y-z|\leq 2r_{\mathcal{B}}+2^{k+1}r_{\mathcal{B}}\leq 2^{k+2}r_{\mathcal{B}}$. Thus
\begin{equation}\label{wang2}
\begin{split}
\int_{2^{k+1}\mathcal{B}}\frac{1}{|y-z|^{n-\alpha}}dy
&=\int_{|y-z|<2^{k+2}r_{\mathcal{B}}}\frac{1}{|y-z|^{n-\alpha}}dy\\
&=\int_0^{2^{k+2}r_{\mathcal{B}}}\int_{\mathbf{S}^{n-1}}\frac{\varrho^{n-1}}{\varrho^{n-\alpha}}\,d\sigma(y')d\varrho\\
&\leq C\big(2^{k+2}r_{\mathcal{B}}\big)^{\alpha}\leq Cm(2^{k+1}\mathcal{B})^{\alpha/n}.
\end{split}
\end{equation}
Then we get
\begin{equation*}
\begin{split}
I_2''\leq C\sum_{k=1}^{\infty}\frac{1}{2^{kn}}\cdot\frac{1}{m(2^{k+1}\mathcal{B})^{1-\alpha/n}}\int_{2\mathcal{B}}|f(z)|dz.
\end{split}
\end{equation*}
Moreover, by \eqref{oftenused} and \eqref{cal1},
\begin{equation*}
\begin{split}
I_2''&\leq C\sum_{k=1}^{\infty}\frac{1}{2^{kn}}\cdot\frac{1}{m(2^{k+1}\mathcal{B})^{1-\alpha/n}}   \big\|f\big\|_{M^{p,\lambda}}m(2\mathcal{B})^{\lambda/{pn}}\cdot m(2\mathcal{B})^{1/{p'}}\\
&\leq C\sum_{k=1}^{\infty}\frac{1}{2^{kn}}\cdot\frac{m(2\mathcal{B})^{1-\alpha/n}}{m(2^{k+1}\mathcal{B})^{1-\alpha/n}}
\big\|f\big\|_{M^{p,\lambda}}\\
&\leq C\big\|f\big\|_{M^{p,\lambda}},
\end{split}
\end{equation*}
where the last inequality follows from the fact that $1-\alpha/n>0$. Summing up the above estimates for $I_2'$ and $I_2''$, we have
\begin{equation*}
I_2\leq C\big\|f\big\|_{M^{p,\lambda}}.
\end{equation*}

\textbf{Estimation of $I_3$:} Note that if $x\in\mathcal{B}$ and $y\in 2^{k+1}\mathcal{B}\setminus 2^k \mathcal{B}$ with $k\in \mathbb{N}$, then
\begin{equation*}
|y-x|\approx|y-x_0|.
\end{equation*}
This fact, together with Lemma \ref{wanglemma2}, implies that for any $x\in \mathcal{B}$,
\begin{equation*}
\begin{split}
&\Big|\mathcal{L}^{-\alpha/2}(f_2)(x)
-e^{-t_{\mathcal{B}}\mathcal{L}}\mathcal{L}^{-\alpha/2}(f_2)(x)\Big|\\
&=\Big|\big(I-e^{-t_{\mathcal{B}}\mathcal{L}}\big)\mathcal{L}^{-\alpha/2}(f_2)(x)\Big|\\
&\leq\int_{(2\mathcal{B})^{\complement}}\big|\widetilde{K}_{\alpha,t_{\mathcal{B}}}(x,y)\big|
\cdot\big|f(y)\big|dy\\
&\leq C\sum_{k=1}^{\infty}\int_{2^{k+1}\mathcal{B}\setminus 2^k \mathcal{B}}
\frac{1}{|x-y|^{n-\alpha}}\cdot\frac{(r_{\mathcal{B}})^2}{|x-y|^{2}}\big|f(y)\big|dy\\
&\leq C\sum_{k=1}^{\infty}\frac{1}{2^{2k}}\cdot\frac{1}{m(2^{k+1}\mathcal{B})^{1-\alpha/n}}
\int_{2^{k+1}\mathcal{B}}\big|f(y)\big|dy.
\end{split}
\end{equation*}
Applying the above pointwise estimate, we thus obtain
\begin{equation*}
\begin{split}
I_3&=\frac{1}{m(\mathcal{B})}\int_{\mathcal{B}}
\big|\mathcal{L}^{-\alpha/2}(f_2)(x)-e^{-t_{\mathcal{B}}\mathcal{L}}\mathcal{L}^{-\alpha/2}(f_2)(x)\big|dx\\
&\leq C\sum_{k=1}^{\infty}\frac{1}{2^{2k}}\cdot\frac{1}{m(2^{k+1}\mathcal{B})^{1-\alpha/n}}
\int_{2^{k+1}\mathcal{B}}\big|f(y)\big|dy.
\end{split}
\end{equation*}
Therefore, by using \eqref{oftenused} and \eqref{cal1},
\begin{equation*}
\begin{split}
I_3&\leq C\sum_{k=1}^{\infty}\frac{1}{2^{2k}}\cdot\frac{1}{m(2^{k+1}\mathcal{B})^{1-\alpha/n}}
\big\|f\big\|_{M^{p,\lambda}}m(2^{k+1}\mathcal{B})^{\lambda/{pn}}\cdot m(2^{k+1}\mathcal{B})^{1/{p'}}\\
&=C\sum_{k=1}^{\infty}\frac{1}{2^{2k}}\big\|f\big\|_{M^{p,\lambda}}\leq C\big\|f\big\|_{M^{p,\lambda}},
\end{split}
\end{equation*}
where the last series is convergent. Combining the above estimates $I_1$, $I_2$ with $I_3$, we obtain the desired result \eqref{main11}, and hence the proof of Theorem \ref{thm1} is complete.
\end{proof}
Let us now turn to prove Theorem $\ref{thm2}$.
\begin{proof}[Proof of Theorem $\ref{thm2}$]
Let $f\in V{M}^{p,\lambda}(\mathbb R^n)$ with $1\leq p<n/{\alpha}$ and $\lambda=n-\alpha p$. Since we endow $\mathrm{VMO}_{\mathcal{L}}(\mathbb R^n)$ with the norm of $\mathrm{BMO}_{\mathcal{L}}(\mathbb R^n)$, by Theorem \ref{thm1}, we obtain
\begin{equation*}
\big\|\mathcal L^{-\alpha/2}(f)\big\|_{\mathrm{VMO}_{\mathcal{L}}}
\leq C\big\|f\big\|_{V{M}^{p,\lambda}},
\end{equation*}
where $C>0$ is independent of $f$. We need only to show that
\begin{equation}\label{main22}
\lim_{r_{\mathcal{B}}\to 0^{+}}\eta(\mathcal L^{-\alpha/2}(f);r_{\mathcal{B}})=0.
\end{equation}
Let $\mathcal{B}=B(x_0,r_{\mathcal{B}})$ be a ball centered at $x_0$ and with radius $r_{\mathcal{B}}$. As before, we decompose the function $f(x)$ as follows:
\begin{equation*}
f(x)=f(x)\cdot\chi_{2\mathcal{B}}+f(x)\cdot\chi_{(2\mathcal{B})^{\complement}}:=f_1(x)+f_2(x),
\end{equation*}
where $2\mathcal{B}=B(x_0,2r_{\mathcal{B}})$ and $(2\mathcal{B})^{\complement}=\mathbb R^n\setminus(2\mathcal{B})$. Thus,
\begin{equation*}
\mathcal L^{-\alpha/2}(f)(x)=\mathcal{L}^{-\alpha/2}(f_1)(x)
+\mathcal{L}^{-\alpha/2}(f_2)(x),
\end{equation*}
and
\begin{equation*}
\begin{split}
e^{-t_{\mathcal{B}}\mathcal{L}}\big(\mathcal L^{-\alpha/2}f\big)(x)
&=e^{-t_{\mathcal{B}}\mathcal{L}}\mathcal{L}^{-\alpha/2}(f_1)(x)+
e^{-t_{\mathcal{B}}\mathcal{L}}\mathcal{L}^{-\alpha/2}(f_2)(x),
\end{split}
\end{equation*}
where $t_{\mathcal{B}}=(r_{\mathcal{B}})^2$ and $r_{\mathcal{B}}$ is the radius of the ball $\mathcal{B}$. Then $\eta(\mathcal L^{-\alpha/2}(f);r_{\mathcal{B}})$ can be written as
\begin{equation*}
\begin{split}
&\eta(\mathcal L^{-\alpha/2}(f);r_{\mathcal{B}})\\
&\leq \sup_{x_0\in\mathbb R^n}\frac{1}{m(\mathcal{B})}\int_{\mathcal{B}}\big|\mathcal{L}^{-\alpha/2}(f_1)(x)\big|dx
+\sup_{x_0\in\mathbb R^n}\frac{1}{m(\mathcal{B})}\int_{\mathcal{B}}\big|e^{-t_{\mathcal{B}}\mathcal{L}}\mathcal{L}^{-\alpha/2}(f_1)(x)\big|dx\\
&+\sup_{x_0\in\mathbb R^n}\frac{1}{m(\mathcal{B})}
\int_{\mathcal{B}}\big|\mathcal L^{-\alpha/2}(f_2)(x)-e^{-t_{\mathcal{B}}\mathcal{L}}\mathcal L^{-\alpha/2}(f_2)(x)\big|dx\\
&=J_1(r_{\mathcal{B}})+J_2(r_{\mathcal{B}})+J_3(r_{\mathcal{B}}).
\end{split}
\end{equation*}

\textbf{Estimation of $J_1$:} By using the estimates \eqref{kernelk} and \eqref{wang1}, we have
\begin{equation*}
\begin{split}
J_1(r_{\mathcal{B}})&\leq\sup_{x_0\in\mathbb R^n}\frac{C}{m(\mathcal{B})}\int_{\mathcal{B}}\bigg\{\int_{2\mathcal{B}}\frac{|f(y)|}{|x-y|^{n-\alpha}}dy\bigg\}dx\\
&\leq\sup_{x_0\in\mathbb R^n}\frac{C}{m(\mathcal{B})^{1-\alpha/n}}
\int_{2\mathcal{B}}|f(y)|dy.
\end{split}
\end{equation*}
Applying H\"older's inequality, we obtain
\begin{equation*}
\begin{split}
\frac{1}{m(\mathcal{B})}\int_{2\mathcal{B}}|f(y)|\,dy
&\leq\frac{1}{m(\mathcal{B})}\bigg(\int_{2\mathcal{B}}|f(y)|^p\,dy\bigg)^{1/p}m(2\mathcal{B})^{1/{p'}}\\
&\leq\frac{C}{m(\mathcal{B})^{1/p}}\big\|f\big\|_{L^p(B(x_0,2r_{\mathcal{B}}))}.
\end{split}
\end{equation*}
Since $s\mapsto\|f\|_{L^p(B(x_0,s))}$ is an increasing function and $m(\mathcal{B})^{\alpha/n}\approx(r_{\mathcal{B}})^{\alpha}$, it is easy to see that
\begin{equation*}
\begin{split}
J_1(r_{\mathcal{B}})
&\leq \sup_{x_0\in\mathbb R^n}C(r_{\mathcal{B}})^{\alpha}\int_{2r_{\mathcal{B}}}^{+\infty}\big\|f\big\|_{L^p(B(x_0,s))}\frac{ds}{s^{n/p+1}}\\
&\leq \sup_{x_0\in\mathbb R^n}C(r_{\mathcal{B}})^{\alpha}\int_{2r_{\mathcal{B}}}^{+\infty}\mathcal{M}_{p,\lambda}(f;x_0,s)\cdot m(B(x_0,s))^{\lambda/{pn}}\frac{ds}{s^{n/p+1}}.\\
\end{split}
\end{equation*}
Note that when $\lambda=n-\alpha p$,
\begin{equation}\label{wang5}
\frac{\,n\,}{p}-\frac{\lambda}{\,p\,}=\alpha,
\end{equation}
and $m(B(x_0,s))^{\lambda/{pn}}\approx s^{\lambda/p}$. Therefore,
\begin{equation*}
\begin{split}
J_1(r_{\mathcal{B}})
&\leq C(r_{\mathcal{B}})^{\alpha}\int_{2r_{\mathcal{B}}}^{+\infty}
\sup_{x_0\in\mathbb R^n}\mathcal{M}_{p,\lambda}(f;x_0,s)\frac{ds}{s^{\alpha+1}}.
\end{split}
\end{equation*}
By a change of variables (here we take $\tau:=s/{2r_{\mathcal{B}}}$), we further obtain
\begin{equation*}
J_1(r_{\mathcal{B}})\leq C\int_1^{+\infty}\sup_{x_0\in\mathbb R^n}\mathcal{M}_{p,\lambda}(f;x_0,2r_{\mathcal{B}}\tau)\frac{d\tau}{\tau^{\alpha+1}}.
\end{equation*}
Since
\begin{equation*}
\sup_{x_0\in\mathbb R^n}\mathcal{M}_{p,\lambda}(f;x_0,2r_{\mathcal{B}}\tau)
\leq \sup_{x_0\in \mathbb R^n,r>0}\mathcal{M}_{p,\lambda}(f;x_0,r)=\big\|f\big\|_{{M}^{p,\lambda}}=\big\|f\big\|_{VM^{p,\lambda}},
\end{equation*}
and $\alpha+1>1$, the following integral is convergent:
\begin{equation*}
\int_1^{+\infty}\frac{d\tau}{\tau^{\alpha+1}}<+\infty.
\end{equation*}
Hence,
\begin{equation*}
\lim_{r_{\mathcal{B}}\to 0^{+}}J_1(r_{\mathcal{B}})=0
\end{equation*}
holds by the Lebesgue dominated convergence theorem.

\textbf{Estimation of $J_2$:} From the estimate \eqref{G}, it follows that
\begin{equation*}
\begin{split}
J_2(r_{\mathcal{B}})
&\leq\sup_{x_0\in\mathbb R^n}\frac{1}{m(\mathcal{B})}\int_{\mathcal{B}}
\bigg\{\int_{2\mathcal{B}}\big|\mathcal{P}_{t_{\mathcal{B}}}(x,y)\big|\cdot\big|\mathcal{L}^{-\alpha/2}(f_1)(y)\big|dy\bigg\}dx\\
&+\sup_{x_0\in\mathbb R^n}\frac{1}{m(\mathcal{B})}\int_{\mathcal{B}}
\bigg\{\sum_{k=1}^{\infty}\int_{2^{k+1}\mathcal{B}\setminus 2^k \mathcal{B}}\big|\mathcal{P}_{t_{\mathcal{B}}}(x,y)\big|\cdot\big|\mathcal{L}^{-\alpha/2}(f_1)(y)\big|dy\bigg\}dx\\
&\leq\sup_{x_0\in\mathbb R^n}\frac{C}{m(2\mathcal{B})}\int_{2\mathcal{B}}\big|\mathcal{L}^{-\alpha/2}(f_1)(y)\big|dy\\
&+\sup_{x_0\in\mathbb R^n}
C\sum_{k=1}^{\infty}\frac{1}{2^{kn}}\cdot\frac{1}{m(2^{k+1}\mathcal{B})}\int_{2^{k+1}\mathcal{B}}\big|\mathcal{L}^{-\alpha/2}(f_1)(y)\big|dy\\
&:=J_2'(r_{\mathcal{B}})+J_2''(r_{\mathcal{B}}).
\end{split}
\end{equation*}
Similarly to the estimate of $J_1(r_{\mathcal{B}})$, we have
\begin{equation*}
\lim_{r_{\mathcal{B}}\to 0^{+}}J_2'(r_{\mathcal{B}})=0.
\end{equation*}
On the other hand, from \eqref{kernelk} and \eqref{wang2}, it follows that
\begin{equation*}
\begin{split}
J_2''(r_{\mathcal{B}})&\leq \sup_{x_0\in\mathbb R^n}
C\sum_{k=1}^{\infty}\frac{1}{2^{kn}}\cdot\frac{1}{m(2^{k+1}\mathcal{B})}\int_{2^{k+1}\mathcal{B}}  \bigg\{\int_{2\mathcal{B}}\frac{|f(z)|}{|y-z|^{n-\alpha}}dz\bigg\}dy\\
&\leq \sup_{x_0\in\mathbb R^n}
C\sum_{k=1}^{\infty}\frac{1}{2^{kn}}\cdot\frac{1}{m(2^{k+1}\mathcal{B})^{1-\alpha/n}}\int_{2\mathcal{B}}|f(z)|dz.
\end{split}
\end{equation*}
Moreover, by H\"{o}lder's inequality and the fact that $2\mathcal{B}\subset 2^{k+1}\mathcal{B}$ with $k\in \mathbb{N}$, we get
\begin{equation*}
\begin{split}
\frac{1}{m(2^{k+1}\mathcal{B})}\int_{2\mathcal{B}}|f(y)|\,dy
&\leq\frac{1}{m(2^{k+1}\mathcal{B})}\bigg(\int_{2\mathcal{B}}|f(y)|^p\,dy\bigg)^{1/p}m(2\mathcal{B})^{1/{p'}}\\
&\leq\frac{1}{m(2^{k+1}\mathcal{B})^{1/p}}\big\|f\big\|_{L^p(B(x_0,2r_{\mathcal{B}}))}.
\end{split}
\end{equation*}
Then we have
\begin{equation*}
J_2''(r_{\mathcal{B}})\leq \sup_{x_0\in\mathbb R^n}
C\sum_{k=1}^{\infty}\frac{1}{2^{kn}}\cdot\frac{1}{m(2^{k+1}\mathcal{B})^{1/p-\alpha/n}}\big\|f\big\|_{L^p(B(x_0,2r_{\mathcal{B}}))}.
\end{equation*}
Note that $1/p-\alpha/n>0$. Hence,
\begin{equation*}
\begin{split}
J_2''(r_{\mathcal{B}})&\leq \sup_{x_0\in\mathbb R^n}
C\sum_{k=1}^{\infty}\frac{1}{2^{kn}}\cdot\frac{1}{m(2\mathcal{B})^{1/p-\alpha/n}}\big\|f\big\|_{L^p(B(x_0,2r_{\mathcal{B}}))}\\
&\leq \sup_{x_0\in\mathbb R^n} C(r_{\mathcal{B}})^{\alpha}\int_{2r_{\mathcal{B}}}^{+\infty}\big\|f\big\|_{L^p(B(x_0,s))}\frac{ds}{s^{n/p+1}},
\end{split}
\end{equation*}
since $\|f\|_{L^p(B(x_0,s))}$ is an increasing function with respect to $s$. Arguing as in the proof of $J_1$, we can also obtain
\begin{equation*}
\lim_{r_{\mathcal{B}}\to 0^{+}}J_2''(r_{\mathcal{B}})=0.
\end{equation*}

\textbf{Estimation of $J_3$:} By Lemma \ref{wanglemma2}, we can deduce that
\begin{equation*}
\begin{split}
J_3(r_{\mathcal{B}})&\leq \sup_{x_0\in\mathbb R^n}\frac{1}{m(\mathcal{B})}
\int_{\mathcal{B}}\bigg\{\int_{(2\mathcal{B})^{\complement}}\big|\widetilde{K}_{\alpha,t_{\mathcal{B}}}(x,y)\big|\cdot\big|f(y)\big|dy\bigg\}dx\\
&\leq\sup_{x_0\in\mathbb R^n}\frac{C}{m(\mathcal{B})}
\int_{\mathcal{B}}\bigg\{\sum_{k=1}^{\infty}\int_{2^{k+1}\mathcal{B}\setminus 2^k \mathcal{B}}
\frac{1}{|x-y|^{n-\alpha}}\cdot\frac{(r_{\mathcal{B}})^2}{|x-y|^{2}}\big|f(y)\big|dy\bigg\}dx\\
&\leq \sup_{x_0\in\mathbb R^n}C\sum_{k=1}^{\infty}\frac{1}{2^{2k}}\cdot\frac{1}{m(2^{k+1}\mathcal{B})^{1-\alpha/n}}
\int_{2^{k+1}\mathcal{B}}\big|f(y)\big|dy.
\end{split}
\end{equation*}
Moreover, by H\"{o}lder's inequality, we obtain that for each $k\in \mathbb{N}$,
\begin{equation*}
\begin{split}
\frac{1}{m(2^{k+1}\mathcal{B})}\int_{2^{k+1}\mathcal{B}}|f(y)|\,dy
&\leq\frac{1}{m(2^{k+1}\mathcal{B})}\bigg(\int_{2^{k+1}\mathcal{B}}|f(y)|^p\,dy\bigg)^{1/p}m(2^{k+1}\mathcal{B})^{1/{p'}}\\
&=\frac{1}{m(2^{k+1}\mathcal{B})^{1/p}}\big\|f\big\|_{L^p(B(x_0,2^{k+1}r_{\mathcal{B}}))}.
\end{split}
\end{equation*}
From this, it follows that
\begin{equation*}
J_3(r_{\mathcal{B}})\leq\sup_{x_0\in\mathbb R^n}C\sum_{k=1}^{\infty}\frac{1}{2^{2k}}\cdot\frac{1}{m(2^{k+1}\mathcal{B})^{1/p-\alpha/n}}
\big\|f\big\|_{L^p(B(x_0,2^{k+1}r_{\mathcal{B}}))}.
\end{equation*}
Since $s\mapsto\|f\|_{L^p(B(x_0,s))}$ is an increasing function, we have
\begin{equation}\label{wang6}
\begin{split}
\int_{4r_{\mathcal{B}}}^{+\infty}\frac{\|f\|_{L^p(B(x_0,s))}}{s^{n/p-\alpha+3}}ds
&=\sum_{k=1}^{\infty}\int_{2^{k+1}r_{\mathcal{B}}}^{2^{k+2}r_{\mathcal{B}}}\frac{\|f\|_{L^p(B(x_0,s))}}{s^{n/p-\alpha+3}}ds\\
&\geq\sum_{k=1}^{\infty}\int_{2^{k+1}r_{\mathcal{B}}}^{2^{k+2}r_{\mathcal{B}}}
\big\|f\big\|_{L^p(B(x_0,2^{k+1}r_{\mathcal{B}}))}\frac{ds}{s^{n/p-\alpha+3}}\\
&\geq C\sum_{k=1}^{\infty}\big\|f\big\|_{L^p(B(x_0,2^{k+1}r_{\mathcal{B}}))}\cdot\frac{1}{(2^{k+1}r_{\mathcal{B}})^{n/p-\alpha+2}}\\
&\geq \frac{C}{(r_{\mathcal{B}})^2}\sum_{k=1}^{\infty}\big\|f\big\|_{L^p(B(x_0,2^{k+1}r_{\mathcal{B}}))}\cdot
\frac{1}{2^{2k}}\cdot\frac{1}{m(2^{k+1}\mathcal{B})^{1/p-\alpha/n}}.
\end{split}
\end{equation}
Therefore, by \eqref{wang5}, \eqref{wang6} and the definition of Morrey norm, we deduce that
\begin{align*}
J_3(r_{\mathcal{B}})&\leq \sup_{x_0\in\mathbb R^n}
C(r_{\mathcal{B}})^2\int_{4r_{\mathcal{B}}}^{+\infty}\frac{\|f\|_{L^p(B(x_0,s))}}{s^{n/p-\alpha+3}}ds\\
&\leq \sup_{x_0\in\mathbb R^n}
C(r_{\mathcal{B}})^2\int_{4r_{\mathcal{B}}}^{+\infty}\frac{\mathcal{M}_{p,\lambda}(f;x_0,s)\cdot s^{{\lambda}/p}}{s^{n/p-\alpha+3}}ds\\
&\leq C(r_{\mathcal{B}})^2\int_{4r_{\mathcal{B}}}^{+\infty}\sup_{x_0\in\mathbb R^n}\mathcal{M}_{p,\lambda}(f;x_0,s)\frac{ds}{s^3}.
\end{align*}
Observe that by a change of variables (here we take $\tau:=s/{4r_{\mathcal{B}}}$),
\begin{equation*}
(r_{\mathcal{B}})^2\int_{4r_{\mathcal{B}}}^{+\infty}\sup_{x_0\in\mathbb R^n}\mathcal{M}_{p,\lambda}(f;x_0,s)\frac{ds}{s^3}
=\int_1^{+\infty}\sup_{x_0\in\mathbb R^n}\mathcal{M}_{p,\lambda}(f;x_0,4r_{\mathcal{B}}\tau)\frac{d\tau}{\tau^3}.
\end{equation*}
Since for any $r_{\mathcal{B}}>0$,
\begin{equation*}
\sup_{x_0\in\mathbb R^n}\mathcal{M}_{p,\lambda}(f;x_0,4r_{\mathcal{B}}\tau)
\leq\sup_{x_0\in \mathbb R^n,r>0}\mathcal{M}_{p,\lambda}(f;x_0,r)=\big\|f\big\|_{{M}^{p,\lambda}}
=\big\|f\big\|_{VM^{p,\lambda}},
\end{equation*}
and the following integral is convergent:
\begin{equation*}
\int_1^{+\infty}\frac{d\tau}{\tau^3}<+\infty.
\end{equation*}
Finally, by using the Lebesgue dominated convergence theorem, we conclude that
\begin{equation*}
\lim_{r_{\mathcal{B}}\to 0^{+}}J_3(r_{\mathcal{B}})=0.
\end{equation*}
Combining the estimates of $J_1$, $J_2$ with $J_3$, we get the desired result \eqref{main22}. Hence, the proof of Theorem \ref{thm2} is complete.
\end{proof}

\section*{Acknowledgments}
The author was partially supported by the grant of Xiangnan University Scientific Research Project ``Real-Variable Theory of Function Spaces and Its Applications".

\bibliographystyle{amsplain}

\end{document}